\documentclass[11pt,letterpaper]{amsart}
\usepackage[margin=1.5in]{geometry}

\usepackage{amsmath, amscd, amsthm,amssymb}
\usepackage{thmtools}
\usepackage{overpic}
\usepackage{microtype}

 \usepackage[pdftex]{color}
 \usepackage[colorlinks,citecolor=cyan,linkcolor=blue]{hyperref}
\usepackage[nameinlink]{cleveref}

\newcommand{\HH}[0]{\mathbb{H}}

\newcommand{\ZZ}{\mathbb{Z}}

\makeatletter
\DeclareRobustCommand{\cev}[1]{%
  {\mathpalette\do@cev{#1}}%
}
\newcommand{\do@cev}[2]{%
  \vbox{\offinterlineskip
    \sbox\z@{$\m@th#1 x$}%
    \ialign{##\cr
      \hidewidth\reflectbox{$\m@th#1\vec{}\mkern4mu$}\hidewidth\cr
      \noalign{\kern-\ht\z@}
      $\m@th#1#2$\cr
    }%
  }%
}
\makeatother
\newtheorem{theorem}{Theorem}%[section]
\newtheorem{lemma}[theorem]{Lemma}

\newtheorem{corollary}[theorem]{Corollary}

\theoremstyle{definition}

\newtheorem{example}[theorem]{Example}

\newcommand{\C}{\mathcal{C}}

\newcommand{\Mod}{\mathrm{Mod}}

\newcommand\tsim{\kern-.4em\sim}

\newcommand{\from}{\colon} % As in f \from S \to S

\renewcommand{\phi}{\varphi}
\renewcommand{\epsilon}{\varepsilon}

\DeclareMathOperator{\lcm}{lcm}

\definecolor{mutedgreen}{rgb}{.1,.75,0.15}
\definecolor{darkgreen}{rgb}{0,0.6,0.4}
\definecolor{purple}{rgb}{0.4,0,0.6}

\begin{document}

\title{There are no periodic Wright maps}

\author{David Futer}
\address{Department of Mathematics\\ 
Temple University}
\email{\href{mailto:dfuter@temple.edu}{dfuter@temple.edu}}

\begin{abstract}
This paper proves that every periodic automorphism of a closed hyperbolic surface $S$ sends some curve to a nearly disjoint curve. In particular, periodic maps cannot have the property that every curve fills with its image, so no such map can give a positive answer to a question of Wright. This paper also answers a question of Schleimer about irreducible periodic surface maps.
\end{abstract}
\date{\today}
\maketitle

\section{Introduction}

This note proves the following result:

\begin{theorem}\label{Thm:Main}
Let $S$ be a closed orientable surface of positive genus. Let $\varphi \in \Mod(S)$ be a periodic mapping class. Then there is an essential simple closed curve $\alpha \subset S$ such that $i(\alpha, \varphi(\alpha)) \leq 1$.
\end{theorem}

Here, $i(\cdot, \cdot)$ denotes geometric intersection number. We refer to Farb and Margalit \cite{FarbMargalit:Primer} for this and other standard terminology for surfaces and mapping class groups. 

Since two simple closed curves that intersect once can only fill a torus or punctured torus, \Cref{Thm:Main} implies

\begin{corollary}\label{Cor:NotFilling}
Let $S$ be a closed orientable surface of genus $g \geq 2$. Then, for every periodic homeomorphism $f \from S \to S$, there is some essential simple closed curve $\alpha$ that does not fill $S$ with its image $f(\alpha)$.
\end{corollary}

\Cref{Cor:NotFilling} is motivated by a question of Wright  \cite[Question 5]{KentLeininger:AtoroidalBundles}. He asked whether there exists a homeomorphism $f \from S \to S$ on a closed hyperbolic surface $S$ such that every curve $\alpha$ fills $S$ with its image $f(\alpha)$, and furthermore $f$ has no fixed points. 
If such a map $f$ exists and is pseudo-Anosov, Wright proved that the embedding $S \hookrightarrow \operatorname{Conf}_2(S)$ given by $x \mapsto (x, f(x))$ is $\pi_1$--injective, and sends every nontrivial loop to a pseudo-Anosov surface braid \cite[Lemma 9]{KentLeininger:AtoroidalBundles}. The resulting surface subgroup would give rise to an atoroidal surface bundle over a surface. If the map $f$ in the construction is periodic, the surface bundle $E$ would have a complex structure, although  \cite[Lemma 9]{KentLeininger:AtoroidalBundles} does not guarantee that it would be atoroidal.

In very recent work, Kent and Leininger \cite{KentLeininger:AtoroidalBundles} used a variant of Wright's construction to build atoroidal surface bundles over surfaces. 
Kent and Leininger also asked whether some variant of Wright's construction, with a periodic map $f$,  might give an atoroidal surface bundle with a complex structure.
By obstructing a periodic answer to Wright's question, \Cref{Cor:NotFilling} illustrates the difficulty in constructing such a bundle.

Wright's question for pseudo-Anosov maps is interesting and open. However, there is a tension between two hypotheses in the question, because
joint work of the author with Aougab and Taylor \cite{AFT:FixedPoints} shows that  a pseudo-Anosov map $f$ with big translation length in the curve graph $\C(S)$ must also have many fixed points. 
%This illustrates the tension between .

There are at least two distinct proofs of \Cref{Thm:Main}. 
The first proof is constructive: it derives structural information about a $\varphi$--invariant polygonal decomposition of $S$, and uses this structure to build an appropriate curve $\alpha$. In certain special cases (see \Cref{Ex:FixedPointFree}), our structural information 
provides an example of a periodic map $f \from S \to S$ that is irreducible and sends an essential closed curve $\alpha$ to a disjoint curve. In the terminology of Schleimer~\cite{Schleimer:StronglyIrreducible}, this map $f$ fails to be strongly irreducible, hence \Cref{Ex:FixedPointFree} provides a positive answer to Schleimer's question \cite[first bullet of Section 5]{Schleimer:StronglyIrreducible}.

An alternate argument, suggested by Wright, is extremely quick. If $S$ is a hyperbolic surface, Nielsen's realization theorem \cite{Nielsen:Realization} implies that there is a hyperbolic metric on which a representative map $f \in \varphi$ acts by isometry. If $\alpha$ is a systole of this metric on $S$, then $f(\alpha)$ is also a systole. But two systoles on a closed surface can intersect at most once (see e.g. \cite[Propositions 3.2 and 3.3]{FanoniParlier:Systoles}).

\subsection*{Acknowledgements}
I thank Autumn Kent, Chris Leininger, and Dan Margalit for an enlightening discussion of Wright's question. I thank Saul Schleimer for pointing out his question \cite[Section 5]{Schleimer:StronglyIrreducible} and for correcting a mistake in an earlier version of \Cref{Lem:SeveralPolygonsCurve}. I thank Alex Wright for a helpful discussion of his question, and for generously sharing a very short alternate proof of \Cref{Thm:Main}.
Finally, I thank the NSF for its support via grant DMS--2405046.

\section{Constructive proof}

As mentioned above, our original proof of \Cref{Thm:Main} is constructive, and provides structural information.
The case of $S = T^2$ serves as a good warm-up.

\subsection{The torus}\label{Sec:Torus}

Every nontrivial periodic element of $\Mod(T^2) \cong SL(2,\ZZ)$ has order 2, 3, 4, or 6; see \cite[Section 7.1.1]{FarbMargalit:Primer}. The unique order--$2$ element ($\varphi = -I$) sends every curve $\alpha$ to itself, with reversed orientation. 

If $\varphi$ has order 4, then a representative homeomorphism $f$ acts by rotation on a square fundamental domain for $T^2$. Thus, if $\alpha$ is a simple closed curve obtained by joining opposite sides of the square, then $i(\alpha, f(\alpha)) = 1$. Similarly, if $\varphi$ has order 3 or 6, then a representative homeomorphism $f$ acts by rotation on a hexagonal fundamental domain. Again, we find a curve $\alpha$ such that $i(\alpha, f(\alpha)) = 1$ by joining opposite sides of the hexagon. This proves \Cref{Thm:Main} for the torus.

% The argument for hyperbolic surfaces is similar in spirit, with the one difference that the single square or hexagon will be replaced by several polygons.

\subsection{Hyperbolic surfaces}
From now on, suppose $S$ is a hyperbolic surface and $\varphi \in \Mod(S)$ is a periodic mapping class. We may assume $\varphi$ is irreducible, as otherwise any curve in a reducing system satisfies $i(\alpha, \varphi(\alpha)) = 0$. By the cyclic case of Nielsen realization (see \cite{Nielsen:Realization} and \cite[Theorem 7.1]{FarbMargalit:Primer}), we fix a hyperbolic structure on $S$ such that a representative homeomorphism $f \in \varphi$ acts on $S$ by isometry.

We will prove that the hyperbolic metric on $S$ can be obtained by gluing together $n$ identical convex polygons in a cyclic fashion,
such that $f$ acts on the polygons by a cyclic permutation. This will allow us to construct an essential simple closed curve $\alpha$ that intersects at most two of the polygons, with at most one arc of intersection with each polygon. It will follow that 
$i(\alpha, \varphi(\alpha)) \leq 1$.

The following fact is well-known.

\begin{lemma}\label{Lem:Turnover}
Let $f \from S \to S$ be a periodic, irreducible isometry.  Then the quotient orbifold $Y = S / \langle f \rangle$ is 
% a turnover
a sphere with three cone points.
\end{lemma}

\begin{proof}
Removing the cone points of $Y$ yields a non-singular punctured surface $Z$ with negative Euler characteristic. If $Z$ were any surface other than a pair of pants, it would contain an essential simple closed curve, whose preimage in $S$ would be a multicurve stabilized by  $f$. But we have assumed that $f$ is irreducible.
\end{proof}

From now on, we assume that $f \from S \to S$ is a periodic irreducible isometry, as in \Cref{Lem:Turnover}.
Thus $Y$ is a sphere with cone points $x_1, x_2, x_3$. Let $p_i$ be the order of the cone point $x_i$, and relabel so that $p_1 \leq p_2 \leq p_3$. Let $G = \langle f \rangle$, and let $r = p_3$.

\begin{lemma}\label{Lem:FixedPoints}
The isometry $f \from S \to S$ has fixed points if and only if $|G| = r$. 
\end{lemma}

\begin{proof}
Any preimage $\widehat x_i \in S$ of a cone point $x_i \in Y$ must have a stabilizer of order $p_i$. Thus the number of distinct preimages of $x_i$ is $|G|/ p_i$. If $|G| = r = p_3$, then $x_3$ has a unique preimage $\widehat x_3$ that is fixed by $f$. 

Conversely, any fixed point of $f$ must project to some cone point $x_i \in Y$.  If $f$ fixes a point $\widehat x_i \in S$, then the full group $G$ stabilizes $\widehat x_i$, hence $|G| = p_i$ for some $i$. Since $p_1 \leq p_2 \leq p_3$, we conclude that $|G| = p_3 = r$.
\end{proof}

The cases $|G| = r$ and $|G| > r$ both occur; see \Cref{Ex:Bolsa,Ex:FixedPointFree}.

\subsection*{Algebraic setup}
Let $\gamma_i$ be a simple closed loop about $x_i$, with all three loops based at a common basepoint $y \in Y$, oriented so that $\gamma_1 \gamma_2 \gamma_3 = 1 \in \pi_1(Y)$. See \Cref{Fig:Notation}, left. Accordingly,
\begin{equation}\label{Eqn:OrbifoldGroup}
\pi_1(Y, y) \cong \langle \gamma_1, \gamma_2, \gamma_3 \: : \: \gamma_i^{p_i} = 1, \: \gamma_1 \gamma_2 \gamma_3 = 1 \rangle.
\end{equation}

Now, $G = \langle f \rangle$ is the deck group of the cover $S \to Y$. By standard covering space theory, there is a surjective homomorphism $\psi \from \pi_1(Y) \to G$ with kernel $\pi_1(S)$. Indeed, each element $g \in \pi_1(Y)$ is realized by a loop in the complement of the cone points, whose path-lift to $S$ defines  the deck transformation 
$\psi(g)$.

\begin{figure}
\begin{overpic}[width=5in]{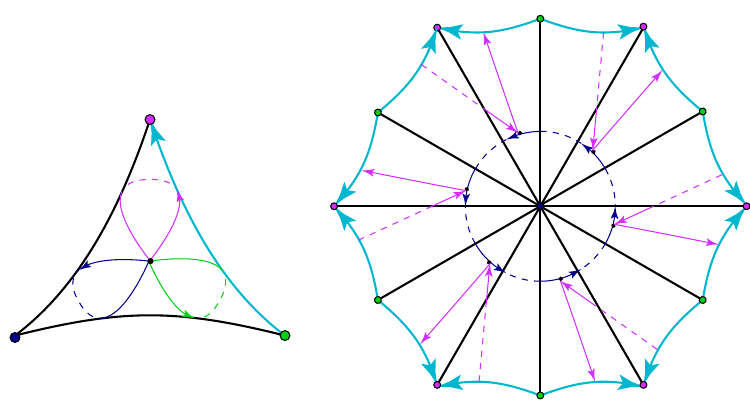}
\put(1,7){$x_3$}
\put(37,7){$x_1$}
\put(19,41){$x_2$}
\put(19,17.5){$y$}
\put(27,24){$\beta$}
\put(24,18.5){$\gamma_1$}
\put(20,25){$\gamma_2$}
\put(13,18.5){$\gamma_3$}
\put(84,38){$\widetilde \gamma_2$}
\put(74,38){$\widetilde \gamma_3$}
\put(77,32){$\widetilde y$}
\end{overpic}
\caption{Left: The quotient orbifold $Y = S / \langle f \rangle$. The notation for curves, arcs, and cone points in $Y$ is used throughout the proof. Right: the reference polygon $P$ in the tiling of  $\HH^2$, with boundary consisting of $2r$ lifts of $\beta \subset Y$. 
% In this example, $r=6$. 
Path-lifts of $\gamma_1$ are not shown.}
\label{Fig:Notation}
\end{figure}
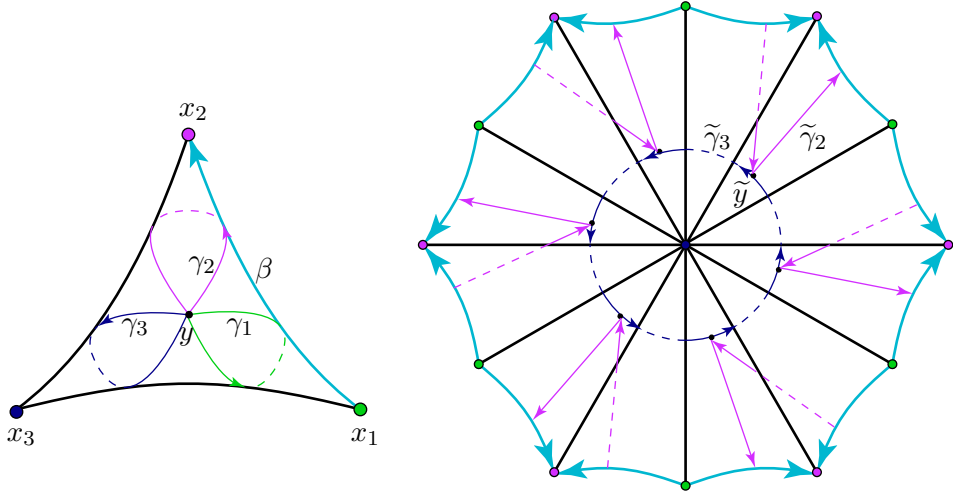

Since $\gamma_3$ has order $p_3 = r$ by \eqref{Eqn:OrbifoldGroup}, and since $\pi_1(S) = \ker(\psi)$ is torsion-free, we learn that $\psi(\gamma_3) \in G$ also has order $r$. Thus $H = G / \langle \psi(\gamma_3) \rangle$ is a cyclic group of order $n = |G| / r$. We define a quotient homomorphism $\nu \from \pi_1(Y) \to H$ by composing $\psi$ with the quotient $G \to G / \langle \psi(\gamma_3) \rangle$. Since $\gamma_2$ and $\gamma_3$ generate $\pi_1(Y)$, and since $\gamma_3 \in \ker(\nu)$, it follows that $\nu(\gamma_2)$ generates $H$. We fix an identification $H \cong \ZZ / n \ZZ$ so that $\nu(\gamma_2) = 1 \! \mod n$.

%    \begin{lemma}\label{Lem:LCM}
%    There is a quotient $\psi \from \pi_1(Y) \to G = \langle f \rangle$ with kernel $\pi_1(S)$.
%    Furthermore,
%    % The cyclic group $G = \langle f \rangle \cong \pi_1(Y) / \pi_1(S)$ has order 
%    \[|G| = \lcm(p_i, p_j) 
%    \quad \text{for any} \quad
%    i \neq j \in \{1,2,3\}.
%    % =\lcm(p_1, p_3) = \lcm(p_2, p_3).
%    \]
%    \end{lemma}
%    
%    \begin{proof}
%    The quotient map $\psi$ comes from standard covering space theory. An element $g \in \pi_1(Y)$ is realized by a loop in the complement of the cone points, whose path-lift to $S$ defines the deck transformation $\psi(g)$.
%    
%    The calculation of $|G|$ is a special case of \cite[Lemma 7.11]{FarbMargalit:Primer}. The proof follows from the observation that $\pi_1(Y)$ is generated by any two of $\gamma_1, \gamma_2, \gamma_3$, hence the cyclic deck group $G  = \langle f \rangle$ is generated by the $\psi$--images of these two loops.
%    \end{proof}

\subsection*{Geometric setup}
The unique hyperbolic metric on $Y = S^2(p_1, p_2, p_3)$ is obtained by doubling the hyperbolic triangle $T$ with angles $\pi/p_1, \pi/p_2, \pi/p_3$. Accordingly, the universal cover $\HH^2 = \widetilde Y$ is tiled by copies of this triangle, with a reflective symmetry in every edge. See \Cref{Fig:Notation}, right.

Let $\beta \subset Y$ be the geodesic arc from $x_1$ to $x_2$ (\Cref{Fig:Notation}, left). Then the complete preimage of $\beta$ in $\HH^2$ is the $1$--skeleton of a $\pi_1(Y)$--equivariant tiling of  $\HH^2$ by convex polygons, with a preimage of $x_3$ in the center of each polygon. 

Let $P \subset \HH^2$ be a reference polygon containing the lifted basepoint $\widetilde y$. 
Then $P$ is obtained by iteratively reflecting the triangle $T$ about a central vertex $\widetilde x_3$ where the triangle has angle $\pi/r$. Thus $\partial P$ is a concatenation of $2r$ edges, labeled $\vec{\beta} \cdot  \cev{\beta} \cdot \vec{\beta} \cdot \cev{\beta} \cdots$. (See \Cref{Fig:Notation}, right.) If $p_1 = 2$, meaning $T$ has a $\pi/2$ angle at $x_1$, then two collinear lifts of $\beta $ are joined to form a single side of $P$, and we think of $P$ as a convex $r$--gon. Otherwise, if $p_1 > 2$, then $P$ is a convex $2r$--gon. In either case, $P$ is strictly convex at each of its ($r$ or $2r$) corners.

\begin{lemma}\label{Lem:CyclicLabels}
Every polygon in the tiling of $\HH^2$ by copies of $P$ can be labeled by an element of $H = \ZZ/n\ZZ$, so that the polygons that share an edge have labels that differ by $1 \! \mod n$. The action of $\pi_1(Y)$ on $\HH^2$ induces an action on the labels by left-translation, with $\pi_1(S)$ acting trivially.
\end{lemma}

\begin{proof}
By construction, the stabilizer of the reference polygon $P$ is the stabilizer of its center point $\widetilde x_3$, which is the cyclic group $\langle \gamma_3 \rangle$.

Every polygon in the tiling has the form $g(P)$, where $g \in \pi_1(Y)$
 is uniquely determined up to pre-composition by some power of $\gamma_3$. Thus we give $g(P)$ the label $\nu(g) \in H$, which is well-defined because $\gamma_3 \in \ker(\nu)$.

Now, suppose $P'$ and $P''$ are adjacent along some edge (a preimage of $\beta$). Since there is a path-lift of $\gamma_2$ dual to every path-lift of $\beta$ (compare \Cref{Fig:Notation}, right), the labels on $P'$ and $P''$ differ by $\nu(\gamma_2) = 1 \! \mod n$.

By construction, the action of $\pi_1(Y)$ on $H$ by left-translation induces an action on the labels by left-translation. Finally, since $\pi_1(S) = \ker(\psi) \subset \ker(\nu)$, any pair of polygons that belong to the same $\pi_1(S)$--orbit also have the same label in $H$.
\end{proof}

We remark that $n = |H|$ might equal $1$ (see \Cref{Lem:OnePolygonCurve}). In this case, all labels are equal, so both $\pi_1(Y)$ and $\pi_1(S)$ act trivially.

\begin{lemma}\label{Lem:SurfacePolygons}
The surface $S$ is obtained by gluing together $n = |G|/r$ isometric copies of $P$ in a cyclic fashion. 
These polygons can be called $P_0, \ldots, P_{n-1}$, so that every $P_i$ is glued only to $P_{i + 1}$ and $P_{i-1}$ (with indices modulo $n$). The map $f$ sends $P_i$ to $P_{i+j}$, for a fixed $j$ that is relatively prime to $n$.
\end{lemma}

\begin{proof}
By construction, the torsion-free subgroup $\pi_1(S) < \pi_1(Y)$ acts freely on the set of polygons in the tiling of $\HH^2$. Thus the interior of each polygon embeds in $S$, and the images of distinct polygons either coincide or have disjoint interiors. Thus $S$ is tiled by isometric copies of $P$.

Every polygon in $S$ contains a lift of $x_3 \in Y$ at its center. In fact, the polygonal tiling of $S$ is the Voronoi tessellation of $S$ with respect to the complete preimage of $x_3$. Since $f$ acts on the lifts of $x_3$ by cyclic permutation, it acts on the polygons in the same manner.
As discussed in \Cref{Lem:FixedPoints}, there are $n = |G|/r$ distinct lifts of $x_3$ to $S$, each with a stabilizer of order $r$, so there are also $n$ distinct polygons.

The remaining conclusions follow from \Cref{Lem:CyclicLabels}. Indeed, $\pi_1(S)$ acts trivially on the $H$--labels of the polygons in $\HH^2$, so every polygon in the tiling of $S$ inherits a label in $H = \ZZ / n\ZZ$.  The transitive left-translation action of $\pi_1(Y)$ on the labels descends to a transitive action of $\langle f \rangle = G = \pi_1(Y)/\pi_1(S)$, because $\pi_1(S)$ acts trivially. Thus $f(P_i) = P_{i+j}$, for a fixed $j$ that is relatively prime to $n$.
Finally, \Cref{Lem:CyclicLabels} also says that adjacent polygons in $\HH^2$ have labels that differ by $1$, hence the same is true in $S$. Thus every $P_i$ must be glued to at least one of $P_{i \pm 1}$. In fact, half the sides of $P_i$ must be glued to $P_{i + 1}$ and the other half to $P_{i-1}$, because of the transitive $f$--action.
\end{proof}

\begin{lemma}\label{Lem:OnePolygonCurve}
Suppose that $r = |G|$. 
Then $S$ can is obtained from the polygon $P_0$ by some side pairing. The map $f$ acts on $P_0$ by rotation about the central vertex $\widehat x_3$. There is an essential simple closed curve $\alpha \subset S$ built by connecting a pair of sides of $P_0$, with the property that $i(\alpha, f(\alpha)) \leq 1$. 
\end{lemma}

\begin{proof}
Since $|G|/r = 1$, \Cref{Lem:SurfacePolygons} implies that $S$ is tiled by a single polygon $P_0$ isometric to $P$. Thus $S$ can be constructed from $P_0$ by some side pairing.
The group $G = \langle f \rangle$ must stabilize $P_0$ and fix its center point
 $\widehat x_3$ (compare \Cref{Lem:FixedPoints}).
 Thus $f$ acts on $P_0$ as a rotation of order $r = |G|$ about $\widehat x_3$.

Let $\alpha \subset S$ be a simple closed curve created by joining two sides  $s, s' \subset P_0$ that are paired in $S$. 
This curve must be essential in $S$: if not, then a lift of $P_0$ to $\HH^2$ would be glued to itself along a preimage of $s$, which contradicts the convexity of the polygon.
Since $f$ acts on $P_0$ by rotation, we conclude that $i(\alpha, f(\alpha)) \leq 1$. 
\end{proof}

\Cref{Lem:OnePolygonCurve} also follows from a theorem of Kulkarni \cite[Theorem 2]{Kulkarni:Periodic}.

% The above case of \Cref{Thm:Main} is directly analogous to the case of the torus.

\begin{lemma}\label{Lem:SeveralPolygonsCurve}
Suppose that $r < |G|$. 
Then there is an essential simple closed curve $\alpha \subset S$ that intersects only two polygons $P_0$ and $P_1$ in the tiling of $S$, and intersects each of them in a simple arc.
Furthermore, $\alpha$ satisfies $i(\alpha, f(\alpha)) \leq 1$. In the special case where $|G| = 2 r$, the curve $\alpha$ satisfies $i(\alpha, f(\alpha)) = 0$.
\end{lemma}

\begin{proof}
By \Cref{Lem:SurfacePolygons}, $S$ is tiled by $n$ isometric polygons 
 $P_0, \ldots, P_{n-1}$, with each $P_i$ glued to $P_{i + 1}$ along half its sides and to $P_{i-1}$ along half its sides. In particular, $P_0$ must have at least two sides (labeled $s,s'$) that are both glued to $P_1$. 
We construct a simple closed curve $\alpha$ such that $\alpha_0 = \alpha \cap P_0$ is a geodesic segment from $s$ to $s'$, and  $\alpha_1 = \alpha \cap P_1$ is a geodesic segment with the same endpoints as $\alpha_0$. 
Now, a path-lift of $\alpha$ to $\HH^2$ must start in some lift $\widetilde P_0$ of $P_0$, run through a lift of $\widetilde P_1$ of $P_1$, and end in another lift $g(\widetilde P_0)$ of $P_0$. Furthermore, 
$g(\widetilde P_0) \neq \widetilde P_0$, because the strictly convex polygons $\widetilde P_0, \widetilde P_1 \subset \HH^2$ cannot meet along two distinct sides. Thus $\alpha$ is essential. 

We emphasize that the above construction works regardless of the choice of sides $s,s' \subset P_0$, provided that both sides are glued to $P_1$. In the special case where $n = 2$, hence $P_0$ is glued to $P_1 = P_{-1}$ along \emph{all} of its sides, we choose $s$ and $s'$ to be adjacent at a corner $ v \in P_0$. We orient and label $s$ and $s'$ so that $P_0$ is to their left, and so that $\vec{s} \cdot \vec{s}\, '$ are concatenated at $v$. 

To prove that $i(\alpha, f(\alpha)) \leq 1$, we consider two cases: $n =2$ and $n \geq 3$. 

If $n = 2$, we claim that $f(s) \notin \{s, s'\}$.  Since the oriented edges $\vec{s}, \vec{s}\, '$ have $P_0$ to their left, they have $P_1$ to their right. By \Cref{Lem:SurfacePolygons}, $f$ interchanges $P_0$ and $P_1$.  Thus, if $f$ maps $s$ to itself, it must reverse the orientation on $s$ and fix its midpoint, contradicting \Cref{Lem:FixedPoints}. Similarly, if $f$ maps $s$ to $s'$, we must have $f(\vec{s}) = \cev{s} '$, hence $f$ maps the terminal vertex of $s$ (namely $v$) to the initial vertex of $s'$ (also $v$), again contradicting \Cref{Lem:FixedPoints}. By an identical argument, $f(s') \notin \{s, s'\}$. Since $\alpha_1 = \alpha \cap P_1$ is a geodesic segment from $s$ to $s'$, and $f(\alpha_0) \subset P_1$ is a geodesic segment between two consecutive sides, neither of which coincides with $s$ or $s'$, we conclude that $f(\alpha_0)$ is disjoint from $\alpha_1$.  
 Similarly, $f(\alpha_1) \subset P_0$ is disjoint from $\alpha_0$. This proves that  $i(\alpha, f(\alpha)) = 0$ when $n = 2$.

Finally, suppose $n \geq 3$. Then \Cref{Lem:SurfacePolygons} implies $f(P_0) = P_j$ and $f(P_1) = P_{j+1}$, for some 
$j \not \equiv 0 \! \mod n$. Consequently, at least one of $P_j$ and $P_{j+1}$ must be distinct from $P_0$ and $P_1$. Thus $\alpha$ and $f(\alpha)$ have at most one polygon in common, hence $i(\alpha, f(\alpha)) \leq 1$.
\end{proof}

Combining 
\Cref{Lem:OnePolygonCurve,Lem:SeveralPolygonsCurve} completes the hyperbolic case of \Cref{Thm:Main}. \qed

\section{Examples}

For this section, we continue to study the situation where $f \from S \to S$ is a periodic, irreducible isometry. By \Cref{Lem:Turnover}, $Y = S / \langle f \rangle = S^2(p_1, p_2, p_3)$, a sphere with three cone points. It is straightforward to construct examples where such a map $f$ has fixed points, and where it does not.

\begin{example}\label{Ex:Bolsa}
Let $S$ be a surface of genus $g \geq 2$, realized as the quotient of a regular $(4g+2)$--gon $P_0$, with opposite sides identified. Let $f \from S \to S$ be a periodic isometry of order $4g+2$, which acts on $P_0$ by a rotation by one click about the center point $\widehat x_3$. Then $S / \langle f \rangle$ is a sphere with cone points $x_1, x_2, x_3$ of order $2, 2g+1, 4g_2$. Indeed, $x_1$ is the quotient of the middle of a side of $P_0$, while $x_2$ is the quotient of the corners, and $x_3$ is the quotient of the center point $\widehat x_3$. Compare \cite[Section 7.2.4]{FarbMargalit:Primer}.

The midpoints of sides of $P$ fall into $2g+1$ distinct $\langle f \rangle$--orbits, and the corners of $P$ fall into two distinct $\langle f \rangle$--orbits, with the generator $f$ permuting the orbits. In particular, $\widehat x_3$ is the only fixed point of $f$.
\end{example}

\begin{example}\label{Ex:FixedPointFree}
Let $Y = S^2(p_1, p_2, p_3)$ be the hyperbolic orbifold with cone points of order
\[
p_1 = 2\cdot 3 = 6, 
\qquad
p_2 = 2\cdot 5 = 10,
\qquad
p_3 = 3\cdot 5 = 15.
\] 
Let $\gamma_i$ be a loop about $x_i$, as in \Cref{Fig:Notation}, so that $\pi_1(Y)$ has the presentation \eqref{Eqn:OrbifoldGroup}. Let  $\psi \from \pi_1(Y) \to G = \ZZ / 30 \ZZ$ be the homomoprphism
\[
\gamma_1 \mapsto 5 \!\mod 30, 
\qquad
\gamma_2 \mapsto -3 \!\mod 30, 
\qquad
\gamma_3 \mapsto -2 \!\mod 30.
\]
In particular, $\psi$ maps every $\gamma_i$ to an element of order $p_i$. Since every torsion element of $\pi_1(Y)$ is conjugate to a power of some $\gamma_i$, the kernel of $\psi$ is torsion-free, and 
$S = \HH^2 / \ker(\psi)$ is a non-singular surface with a degree 30 cyclic cover $S \to Y$. The deck group $G$ is generated by a map $f \from S \to S$ of order 30. An Euler characteristic calculation shows that $S$ has genus $11$. 

By \Cref{Lem:FixedPoints}, $f$ has no fixed points.
By \Cref{Lem:SurfacePolygons}, $S$ is obtained as the union of two $30$--gons $P_0$ and $P_1$ that are interchanged by $f$. By \Cref{Lem:SeveralPolygonsCurve}, there is an essential simple closed curve $\alpha \subset S$ that intersects each $P_i$ in an arc, such that $i(\alpha, f(\alpha)) = 0$. In particular, this map $f$ is irreducible but not strongly irreducible, producing an affirmative answer to Schleimer's question \cite[first bullet of Section 5]{Schleimer:StronglyIrreducible}.
\end{example}

More generally, if $Y = S / \langle f \rangle = S^2(p_1, p_2, p_3)$, then \cite[Lemma 7.11]{FarbMargalit:Primer} shows that
\[|G| = \lcm(p_i, p_j) 
\quad \text{for any} \quad
i \neq j \in \{1,2,3\}.
\]
If $f$ has no fixed points, meaning $p_i < |G|$ for each $i$, then a $\lcm$ analysis shows that $|G|$ has at least three distinct prime factors. It follows that \Cref{Ex:FixedPointFree}, with a surface $S$ of genus 11 and a map $f \from S \to S$ of order $30 = 2 \cdot 3 \cdot 5$, is the smallest example of a periodic, irreducible, fixed-point-free isometry.

\bibliographystyle{amsplain}
\bibliography{biblio}

\end{document}